\documentclass[11pt]{article}
\usepackage{amsmath,amsthm,amssymb,amscd,epsfig,url}
\usepackage{fullpage}
\usepackage{times}
\newtheorem{theorem}{Theorem}[section]
\newtheorem{lemma}[theorem]{Lemma}
\newtheorem{proposition}[theorem]{Proposition}
\newtheorem{corollary}[theorem]{Corollary}

\newtheorem{definition}[theorem]{Definition}

\def\<{{\langle}} \def\>{{\rangle}}       
\begin{document}

\author{Aaron Potechin \\
MIT\\
potechin@mit.edu}

\title{A Note on a Problem of Erd\H{o}s and Rothschild}
\maketitle
\begin{abstract}
A set of $q$ triangles sharing a common edge is a called a book of size $q$. Letting $bk(G)$ denote the size of the largest book in a graph $G$, Erd\H{o}s and Rothschild \cite{erdostwo} asked what the minimal value of $bk(G)$ is for graphs $G$ with $n$ vertices and 
a set number of edges where every edge is contained in at least one triangle. In this paper, we show that for any graph $G$ with $n$ 
vertices and $\frac{n^2}{4} - nf(n)$ edges where every edge is contained in at least one triangle, 
$bk(G) \geq \Omega\left(\min{\{\frac{n}{\sqrt{f(n)}}, \frac{n^2}{f(n)^2}\}}\right)$.
\end{abstract}

\thispagestyle{empty}
\noindent \textbf{Acknowledgement:}\\
This material is based on work supported by the National Science Foundation Graduate Research Fellowship 
under Grant No. 0645960.
\newpage
\section{Introduction}\label{intro}
A set of $q$ triangles sharing a common edge is called a book of size $q$. Erd\H{o}s \cite{erdosone} started the study of books in graphs and this study has since attracted a great deal of attention in extremal graph theory (see e.g. \cite{bollobas}, \cite{moreerdosone}, \cite{moreerdostwo}, \cite{russianpaper}) and graph Ramsey theory (see e.g. \cite{ramseyone}, \cite{ramseytwo}, \cite{ramseythree}, \cite{ramseyfour}, \cite{ramseyfive}, \cite{ramseysix}, \cite{ramseyseven}).

Erd\H{o}s and Rothschild \cite{erdostwo} considered the problem of bounding $h(n,c)$, which is defined as follows.
\begin{definition} \noindent
\begin{enumerate}
\item Let $bk(G)$ denote the size of the largest book in a graph $G$.
\item Let $h(n,c)$ denote the minimum value of $bk(G)$ over all graphs on $n$ vertices with more than $cn^2$ edges such that every edge is contained in at least one triangle. 
\end{enumerate}
\end{definition}
This problem received considerable attention (see e.g. the Erd\H{o}s problem papers \cite{erdostwo}, \cite{erdosthree}, \cite{erdosfour} and the book \cite{erdosbook}). In terms of lower bounds, Szemer\'{e}di used his regularity lemma to show 
that for all fixed $c < \frac{1}{4}$, $h(n,c) \to \infty$ as $n \to \infty$. Ruzsa and Szemer\'{e}di \cite{ruzsa} further 
showed that this implies Roth's theorem, that every subset of $[1,n]$ without a 3-term arithmetic progression has size $o(n)$, 
as well as the (6,3)-theorem which states that every 3-uniform hypergraph on $n$ vertices in which the union of the endpoints of 
any 3 edges has at least 6 vertices must have $o(n^2)$ edges. Fox (see the end of the introduction of \cite{fox}) recently strengthened the quantitative bounds on $h(n,c)$ to be $2^{\Omega(\log^{*}{n})}$ rather than $(\log^{*}{n})^{\Omega(1)}$. For the case $c \geq \frac{1}{4}$, Edwards \cite{edwards} and Khad\u{z}iivanov, Nikiforov \cite{russianpaper} independently showed that $h(n,c) \geq \frac{n}{6}$ for all $c \geq \frac{1}{4}$.

In terms of upper bounds, Alon and Trotter (see \cite{erdosfour}) showed that for any $c < \frac{1}{4}$, $h(n,c)$ is $O(\sqrt{n})$. Fox 
and Loh \cite{fox} recently strengthened this to show that for any fixed $c < \frac{1}{4}$, $h(n,c)$ is $n^{O(\frac{1}{\log{\log{n}}})}$. Thus, there is a threshold for this problem at $c = \frac{1}{4}$.
\subsection{Previous Work and Our Results}\label{results}
In this paper we examine what happens at this threshold by posing the problem as follows.
\begin{definition}
Given a function $f: \mathbb{Z^{+}} \to \mathbb{R^{+}}$, define $\gamma(n,f)$ to be the minimal value of $bk(G)$ over all graphs with $n$ vertices and at least $\lceil\frac{n^2}{4} - nf(n)\rceil$ edges such that every edge is contained in at least one triangle.
\end{definition}
Bollob\'{a}s and Nikiforov \cite{bollobas} showed the following bounds. For any $\epsilon > 0$ and $0 < c < \frac{2}{5}$, 
if $f(n)$ is $\Theta(n^c)$ for all $n$ then for all sufficiently large $n$, $(1 - \epsilon)\frac{n}{2\sqrt{2f(n)}} < \gamma(n,f) < 
(1 + \epsilon)\frac{n}{2\sqrt{2f(n)}}$. In fact, the upper bound comes from a graph described by Erd\H{o}s \cite{erdosthree} and applies whenever $f(n)$ is $\Theta(n^c)$ for any $c \in (0,1)$.

We extend the lower bounds of Bollob\'{a}s and Nikiforov \cite{bollobas} by obtaining the 
following bounds on $\gamma(n,f)$
\begin{theorem}\label{mainresult}
If $G$ is a graph with exactly $\frac{n^2}{4} - nf(n)$ edges where each edge is contained in at least one triangle and $f(n) \leq \frac{n}{1000}$ then either $b(G) > \frac{n}{1000}$ or $f(n)(f(n) + bk(G))bk(G) \geq \frac{n^2}{1250}$
\end{theorem}
\begin{corollary}
If $\frac{n^2}{4} - nf(n)$ is an integer and $f(n) \leq \frac{n}{1000}$ then 
$\gamma(n,f) \geq \min\{\frac{n}{50\sqrt{f(n)}}, \frac{n^2}{2500{f(n)^2}}, \frac{n}{1000}\}$
\end{corollary}
\begin{proof}
For any graph $G$, if $bk(G) \leq \frac{n}{1000}$ then either $bk(G) \geq f(n)$ or $bk(G) \leq f(n)$. In the first case, 
$2f(n)bk(G)^2 \geq f(n)(f(n) + bk(G))bk(G) \geq \frac{n^2}{1250}$ so $bk(G) \geq \frac{n}{50\sqrt{f(n)}}$. In the second case, 
$2{f(n)^2}bk(G) \geq f(n)(f(n) + bk(G))bk(G) \geq \frac{n^2}{1250}$ so $bk(G) \geq \frac{n^2}{2500{f(n)^2}}$
\end{proof}
\begin{corollary}
For any $c \in (0,1)$, if $f(n)$ is $\Theta(n^c)$ for all $n$ then $\gamma(n,f)$ is $\Theta(n^{1 - \frac{c}{2}})$ if $c \leq \frac{2}{3}$ 
and $\gamma(n,f)$ is $\Omega(n^{2 - 2c})$ if $c \geq \frac{2}{3}$.
\end{corollary}
\subsection{Proof sketch}\label{sketch}
The idea behind our bound is as follows. If both $f(n)$ and $bk(G)$ are small then roughly speaking 
$G$ will have the following structure. It will consist of a large set of vertices of high degree 
(degree at least $\frac{n}{2} - O(f(n) + bk(G))$) and a set of at most $O(f(n))$ vertices with low degree (degree at most $O(f(n) + bk(G))$) 
where the induced subgraph on the vertices of high degree is bipartite. The bound now follows from counting the 
number of triangles in the graph such that two of its vertices are high degree and one vertex is low degree.

On the one hand, there are $\Omega(n^2)$ edges between vertices of high degree. Each such edge is contained in a triangle 
and the third vertex must be low degree. Thus we must have at least $\Omega(n^2)$ such triangles. On the other hand, each 
such triangle must contain two edges between a low degree vertex and a high degree vertex. 
The number of such edges is at most the number of low degree vertices (which is at most $O(f(n))$) times the degree of these vertices 
(which is at most $O(f(n) + bk(G))$). Each such edge appears in at most $bk(G)$ triangles so we have that there are 
$O(f(n)(f(n) + bk(G))bk(G))$ such triangles and the result follows.

Unfortunately, the graph $G$ may not quite have this structure. The main difficulty is in showing that the structure of $G$ is 
close to the structure described above and that this is sufficient to prove our bounds.
\section{Proof of Thoerem \ref{mainresult}}\label{proof}
Throughout this section, we will assume that $G$ is a graph with $n$ vertices and $(\frac{1}{4} - \frac{f(n)}{n})n^2$ edges 
such that every edge of $G$ is in at least one triangle yet $bk(G)$ is small. The structure of of $G$ is largely 
determined by the following lemma.
\begin{lemma}\label{keylemma}
If $T$ be a triangle in $G$ where the vertices have degrees $d_1,d_2,d_3$ then $bk(G) \geq \frac{d_1 + d_2 + d_3 - n}{3}$
\end{lemma}
\begin{proof}
If $v$ is a vertex not in $T$, let $deg_T(v)$ be the number of vertices in $T$ which are adjacent to $v$. Now let $x$ be the number 
of triangles excluding $T$ which share an edge with $T$.
$$x = \sum_{v \notin T}{{{deg_T(v)} \choose 2}} \geq \sum_{v \notin T}{(deg_T(v) - 1)} = (\sum_{v \notin T}{deg_T(v)}) - (n-3) = 
d_1 + d_2 + d_3 - 6 - (n-3)$$
By the pigeonhole principle, one edge of $T$ must be contained in at least $\frac{x}{3}$ triangles excluding $T$, so \\
$bk(G) \geq \frac{x}{3} + 1 = \frac{d_1 + d_2 + d_3 - n}{3}$, as needed.
\end{proof}
To analyze $G$, we begin by splitting the vertices of $G$ into two parts depending on their degree. 
\begin{definition} \noindent
\begin{enumerate}
\item Let $V_H$ be the set of vertices with degree greater than $\frac{2n}{5}$ and let $n_h = |V_H|$
\item Let $V_L$ be the set of vertices with degree at most $\frac{2n}{5}$ and let $n_l = |V_L|$
\end{enumerate}
\end{definition}
We now show that under our assumptions $n_l$ is small with the following lemma.
\begin{lemma}\label{fewlowdegreevertices}
If $f(n) \leq \frac{n}{1000}$ and $bk(G) \leq \frac{n}{1000}$ then $n_l \leq 20f(n)$.
\end{lemma}
\begin{proof}
Applying the following structural graph theorem proved by Andrasfai, Erdos, and Sas \cite{structuralresult} with $r = 2$, the induced 
subgraph of $G$ on the set of vertices $V_H$ must either contain a triangle (in which case $bk(G) > \frac{n}{15}$ 
by Lemma \ref{keylemma}) or it must be bipartite:
\begin{theorem}
If $G$ is a $K_{r+1}$-free graph with $n$ vertices and minimal degree greater than $(1 - \frac{3}{3r-1})n$ then $G$ is 
$r$-colorable
\end{theorem}
Now lets count up the total number of possible edges in $G$. There are at most $\frac{(n_h)^2}{4} = \frac{(n - n_l)^2}{4}$ edges 
between vertices in $V_H$ and there are at most $\frac{2n}{5}n_l$ edges containing a vertex in $V_L$. This gives a total 
of $\frac{n^2}{4} - \frac{n}{10}n_l + \frac{(n_l)^2}{4}$ possible edges. However, by definition $G$ has exactly 
$(\frac{1}{4} - \frac{f(n)}{n})n^2$ edges. 

Thus we have that $\frac{n^2}{4} - \frac{n}{10}n_l + \frac{(n_l)^2}{4} \geq (\frac{1}{4} - \frac{f(n)}{n})n^2$. 
Multiplying this equation by $-\frac{10}{n}$ and rearranging the terms we have that $n_l - \frac{5(n_l)^2}{2n} \leq 10f(n)$. 
If $n_l < \frac{n}{5}$ then the result follows, so we just need to prove that $n_l < \frac{n}{5}$. For this, we will use 
the following lemma, which is a weaker but simpler version of Theorem 1 of Bollobas and Nikiforov \cite{bollobas} which is 
sufficient for our purposes.
\begin{lemma}
For any graph $G$, $\sum_{v \in V(G)}{d(v)^2} \leq |E(G)|(n + bk(G))$
\end{lemma}
\begin{proof}
Note that 
$$\sum_{v \in V(G)}{d(v)^2} = \sum_{v \in V(G)}{\sum_{e = \{v_1,v_2\} \in E(G): \atop v \in \{v_1,v_2\}}{d(v)}} = 
\sum_{e = \{v_1,v_2\} \in E(G)}{(d(v_1) + d(v_2))}$$ 
We bound this with the following proposition.
\begin{proposition}
For any edge $e = \{v_1,v_2\} \in E(G)$, $d(v_1) + d(v_2) \leq n + bk(G)$
\end{proposition}
\begin{proof}
Given an edge $e = \{v_1,v_2\}$, for each i $\in \{0,1,2\}$ let $m_i$ be the number of vertices in 
$V(G) \setminus \{v_1,v_2\}$ which are adjacent to $i$ vertices in $\{v_1,v_2\}$. We have the following equations
\begin{enumerate}
\item $m_0 + m_1 + m_2 = n - 2$
\item $d(v_1) + d(v_2) - 2 = m_1 + 2m_2$
\end{enumerate}
Subtracting the first equation from the second equation gives $m_2 - m_0 = d(v_1) + d(v_2) - n \leq m_2 \leq bk(G)$. Thus, 
$d(v_1) + d(v_2) \leq n + bk(G)$.
\end{proof}
Using this proposition, 
$\sum_{v \in V(G)}{d(v)^2} \leq |E(G)|(n + bk(G))$, as needed.
\end{proof}
We can now prove that $n_l < \frac{n}{5}$. Consider the quantity $\sum_{v \in V(G)}{(d(v) - \frac{2|E(G)|}{n})^2}$. On one hand, 
\begin{align*}
\sum_{v \in V(G)}{\left(d(v) - \frac{2|E(G)|}{n}\right)^2} &= 
\sum_{v \in V(G)}{\left(d(v)^2 - 4d(v)\frac{|E(G)|}{n} + \frac{4|E(G)|^2}{n^2}\right)} = \sum_{v \in V(G)}{d(v)^2} - \frac{4|E(G)|^2}{n} \\
&\leq (n + bk(G) - \frac{4|E(G)|}{n})|E(G)| = (bk(G) + 4f(n))|E(G)| \\
&\leq \frac{n^2}{4}(bk(G) + 4f(n)) \leq \frac{n^3}{800}
\end{align*}
On the other hand, if $n_l \geq \frac{n}{5}$ then we already have $\frac{n}{5}$ vertices with degree less than $\frac{2n}{5}$ so 
\begin{align*}
\sum_{v \in V(G)}{(d(v) - \frac{2|E(G)|}{n})^2} &\geq \frac{n}{5}(\frac{2n}{5} - \frac{n}{2} + 2f(n))^2 \\
&= \frac{n}{5}(\frac{n}{10} - 2f(n))^2 \\
&> \frac{n^3}{800}
\end{align*}
This is a contradiction, which completes the proof that $n_l < \frac{n}{5}$ and thus the proof of the lemma.
\end{proof}
Now that we have shown Lemma \ref{fewlowdegreevertices}, we are ready to prove our bound.
\begin{theorem}
If $f(n) \leq \frac{n}{1000}$ and $bk(G) \leq \frac{n}{1000}$ then $f(n)(f(n) + bk(G))bk(G) \geq \frac{n^2}{2000}$
\end{theorem}
\begin{proof}
The idea of the proof is to count the number of triangles satisfying the following property:
\begin{definition}
We call a triangle of $G$ well-behaved if it contains two vertices in $V_H$ and one vertex in $V_L$ 
of degree at most $5(f(n) + bk(G))$.
\end{definition}
To lower bound the number of such triangles, we count the number of edges between vertices of $V_H$ 
satisfying the following property
\begin{definition}
Call an edge $e = \{v_1,v_2\}$ between two vertices of $V_H$ well-behaved if \\
$d(v_1) + d(v_2) \geq n - 2bk(G) - 5f(n)$
\end{definition}
\begin{proposition}
There are at least as many well-behaved triangles as there are well-behaved edges.
\end{proposition}
\begin{proof}
By our assumptions about $G$, every well-behaved edge of $G$ must be in a triangle and by 
Lemma \ref{keylemma} this triangle must be well-behaved.
\end{proof}
\begin{lemma}
There are at least $\frac{n^2}{25}$ well-behaved edges.
\end{lemma}
\begin{proof}
Let $x$ be the number of well-behaved edges and consider the expression \\
$\sum_{e = \{v_1,v_2\} \in E(G)}{(d(v_1) + d(v_2) - n + 4f(n) + bk(G))}$ \\
On the one hand, using the Cauchy-Schwarz inequality, 
\begin{align*}
\sum_{e = \{v_1,v_2\} \in E(G)}{(d(v_1) + d(v_2) - n + 4f(n) + bk(G))} &= 
\sum_{v \in V(G)}{d(v)^2} - (n - 4f(n) - bk(G))|E(G)| \\
&\geq \frac{1}{n}\left(\sum_{v \in V(G)}{d(v)}\right)^2 - (n - 4f(n) - bk(G))|E(G)| \\
&= 4\frac{|E(G)|^2}{n} - (n - 4f(n) - bk(G))|E(G)| \\
&= (n - 4f(n))|E(G)| - (n - 4f(n) - bk(G))|E(G)| \\
&\geq 0
\end{align*}
On the other hand, since the maximum degree of a vertex in $v_L$ is $\frac{n}{5}$, there are at most $\frac{2n}{5}n_l \leq 8nf(n) \leq \frac{n^2}{125}$ edges which are not between vertices of $V_H$. Now for any edge between vertices in $V_H$ which is not 
well-behaved, it will make a negative contribution to this sum of at least $f(n) + bk(G)$, which gives a total negative 
contribution of at least $(\frac{n^2}{4} - nf(n) - x - \frac{n^2}{125})(f(n) + bk(G))$. For any edge $e = \{v_1,v_2\}$, $d(v_1) + d(v_2) - n \leq bk(G)$ so the 
contribution to this sum from $e$ is at most $4(f(n) + bk(G))$. Thus, the total positive contribution to the sum is 
at most $4(x + \frac{n^2}{125})(f(n) + bk(G))$. Putting everything together we have that 
\begin{align*}
\sum_{e = \{v_1,v_2\} \in E(G)}&{(d(v_1) + d(v_2) - n + 4f(n) + bk(G))} \\
&\leq -(\frac{n^2}{4} - nf(n) - x - \frac{n^2}{125})(f(n) + bk(G)) + 4(x + \frac{n^2}{125})(f(n) + bk(G)) \\
&\leq (5x + nf(n) - \frac{n^2}{4} + \frac{n^2}{25})(f(n) + bk(G))
\end{align*}
From before, this must be nonnegative so we have that 
$x \geq \frac{n^2}{5}(\frac{1}{4} - \frac{1}{25} - \frac{1}{1000}) > \frac{n^2}{25}$, as needed.
\end{proof}
Since there are at least as many well-behaved triangles as there are edges, $G$ must contain at least $\frac{n^2}{25}$ 
well-behaved triangles. We now complete the proof of Theorem \ref{mainresult} by upper bounding the number of well-behaved triangles. Note that every 
well-behaved triangle must contain two edges incident with a vertex of degree at most 
$5(f(n) + bk(G))$. However, there are at most $n_l \leq 20f(n)$ such vertices, so there can be at most 
$100f(n)(f(n) + bk(G))$ such edges. By definition, each such edge can only appear in $bk(G)$ triangles so 
there can be at most $\frac{100f(n)(f(n) + bk(G))bk(G)}{2} = 50f(n)(f(n) + bk(G))bk(G)$ well-behaved triangles. 

Putting everything together, 
$50f(n)(f(n) + bk(G))bk(G) \geq \frac{n^2}{25}$ so we obtain that \\
$f(n)(f(n) + bk(G))bk(G) \geq \frac{n^2}{1250}$, as 
claimed.
\end{proof}
\section{Conclusion}\label{conclusion}
In this paper, we improved the lower bounds on $\gamma(n,f)$ whenever $f$ is $\Theta(n^c)$ for all 
$c \in (\frac{2}{5},1)$. With these results, we now know $\gamma(n,f)$ to within a constant factor whenever 
$f$ is $\Theta(n^c)$ and $c \in [0,\frac{2}{3}]$. However, this raises several new questions. First, can we merge 
the upper bounds of Fox, Loh \cite{fox} and Bollob\'{a}s, Nikiforov \cite{bollobas} to obtain improved upper bounds when 
$f$ is $\Theta(n^c)$ and $c \in (\frac{2}{3},1)$? Second, can we merge our new lower bounds with the lower bounds of 
Fox \cite{fox}? We have a long ways to go before we fully understand $h(n,c)$ and $\gamma(n,f)$.

\end{document}